\documentclass[a4paper, 12pt]{article}
\usepackage{amssymb,latexsym} \author{Johannes Sj\"ostrand\footnote{Ce travail a b\'en\'efici\'e d'une aide de l'Agence Nationale de la Recherche
portant la r\'ef\'erence ANR-08-BLAN-0228-01}\\
\small IMB, Universit\'e de Bourgogne\\
\small 9, Av. A. Savary, BP 47870\\
\small FR-21078 Dijon C\'edex\\
\footnotesize  and UMR 5584, CNRS\\
\footnotesize johannes@u-bourgogne.fr
} \date{Dedicated to V.G.~Maz'ya}
\title{Resolvent estimates for
  non-self-adjoint operators via semi-groups}

\newtheorem{dref}{Definition}[section] \newtheorem{lemma}[dref]{Lemma}
\newtheorem{theo}[dref]{Theorem} \newtheorem{prop}[dref]{Proposition}

\newenvironment{proof}{\par\noindent{{\bf Proof.}}}{\hfill$\Box$
\medskip} 
\newenvironment{proofof}{\par\noindent{{\bf Proof} of }}{\hfill$\Box$
\medskip}

\newcommand{\ekv}[2]{\begin{equation}\label{#1}#2\end{equation}}
\newcommand{\eekv}[3]{\begin{eqnarray}\label{#1}#2 \\ #3
\nonumber\end{eqnarray}}

  \newcommand\iint{\int\hskip -2mm\int}

 \newcommand\trans[1]{{^t\hskip -2pt
#1}}
\begin{document}

\maketitle
\begin{abstract} We consider a non-self-adjoint $h$-pseudodifferential operator $P$ in the semi-classical limit ($h\to 0$). If $p$ is the leading symbol, then under suitable assumptions about the behaviour of $p$ at infinity, we know that the resolvent $(z-P)^{-1}$ is uniformly bounded for $z$ in any compact set not intersecting the closure of the range of $p$. Under a sub\-ellipticity condition, we show that the resolvent extends locally inside the range up to a distance ${\cal O}(1)((h\ln \frac{1}{h})^{k/(k+1)})$ from certain boundary points, where $k\in \{ 2,4,...\}$. This is a slight improvement of a result by Dencker, Zworski and the author, and it has recently been obtained by W. Bordeaux Montrieux in a model situation where $k=2$. The method of proof is different from the one of Dencker et al, and is based on estimates of an associated semi-group.

  \medskip \par \centerline{\bf R\'esum\'e} Nous consid\' erons un op\' erateur $h$-pseudodiff\' erentiel non-autoadjoint $P$ dans la limite semi-classique ($h\to 0$). Si $p$ d\' esigne le symbole principal, alors sous des hypoth\`eses convenables sur le comportment de $p$ \`a l'infini nous savons que la r\' esolvante $(z-P)^{-1}$ est uniform\' ement born\' ee pour $z$ dans un compact qui ne rencontre pas l'adh\' erence de l'image de $p$. Sous une hypoth\`ese des sous-ellipticit\' e, nous montrons que la r\' esolvante s'\' etend vers l'int\' erieur de cet image jusqu'\`a une distance ${\cal O}(1)((h\ln \frac{1}{h})^{k/(k+1)})$ de certains points du bord, o\`u $k\in \{ 2,4,...\}$. Ceci est une l\'eg\`ere am\'elioration d'un r\' esultat de Dencker, Zworski et l'auteur. Cette am\' elioration a \' et\' e obtenue r\' ecemment par W. Bordeaux Montrieux dans une situation mod\`ele o\`u $k=2$. La m\' ethode de preuve, qui est diff\' erente de celle de Dencker et al, est bas\' ee sur des estimations sur un semi-groupe microlocal associ\' e.  \end{abstract}

\tableofcontents

\section{Introduction}\label{in}
\setcounter{equation}{0}

\par In this paper, we are interested in bounds on the resolvent $(z-P)^{-1}$ of a non-self-adjoint $h$-pseudodifferential operator with leading symbol $p$ when $h\to 0$, for $z$ in a neighborhood of certain points on the boundary of the range of $p$. The interest in such questions arouse with that in pseudospectra of non-self-adjoint operators, see \cite{Tr, TrEm}. Under reasonable hypothesies we know that $(z-P)^{-1}$ is uniformly bounded for $h>0$ small enough and for $z$ in any fixed compact set in ${\bf C}$, disjoint from the closure of the range of $p$. On the other hand, by a quasi-mode construction of E.B.~Davies \cite{Da}, that was generalized by Zworski \cite{Zw} by reduction to an old quasi-mode construction of H\"ormander (see also \cite{DeSjZw} for a more direct approach), we also know that if ${\bf C}\ni z=p(\rho )$, where $\rho $ is a point in phase space where $i^{-1}\{p,\overline{p}\}>0$ and $\{\cdot ,\cdot \cdot \}$ denotes the Poisson bracket, then we have quasimodes for $P-z$ in the sense that there exist $u=u_h\in C_0^\infty $, normalized in $L^2$, such that the $L^2$ norm of $(P-z)u_h$ is ${\cal O}(h^\infty )$, implying, somewhat roughly, that the norm of the resolvent (whenever it is defined) cannot be bounded by a negative power of $h$.  

\par A natural question is then what happens when $z$ is close to the boundary of the range of $p$. L.~Boulton \cite{Bo} and Davies \cite{Da2} obtained some results about this in the case of the non-self-adjoint harmonic operator on the real line. As with the quasi-mode construction this question is closely related to classical results in the general theory of linear PDE, and with N.~Dencker and Zworski (\cite{DeSjZw}) we were able to find quite general results closely related to the classical topic of subellipticity for pseudodifferential operators of principal type, studied by Egorov, H\"ormander and others. See \cite{Ho}. This topic in turn is closely related to the oblique derivate problem and degenerate elliptic operators, where V.G.~Maz'ya has made important contributions. See \cite{Maz1, Maz2}.

In \cite{DeSjZw} we obtained resolvent estimates at certain boundary points,\\
(A) under a non-trapping condition,\\
and\\
(B) under a stronger ``subellipticity condition''.

\par In case (A) we could apply quite general and simple arguments related to the propagation of regularity and in case (B) we were able to adapt general Weyl-H\"ormander calculus and H\"ormander's treatment of subellipticity for operators of principal type (\cite{Ho}). In the first case we obtained that the resolvent extends and has temperate growth in $1/h$ in discs of radius ${\cal O}(h\ln 1/h)$ centered at the appropriate boundary points, while in case (B) we got the corresponding extension up to distance ${\cal O}(h^{k/(k+1)})$, where the integer $k\ge 2$ is determined by a condition of ``subellipticity type''. 

\par However, the situation near boundary points of the type (B) is more special than the general subellipticity situations considered by Egorov and H\"ormander, and the purpose of the present paper is to develop such an approach by studying an associated semi-group basically as a Fourier integral operator with complex phase in the spirit of Maslov \cite{Mas}, Kucherenko \cite{Ku}, Melin-Sj\"ostrand \cite{MeSj76}. (See also the more recent works by A.~Menikoff-Sj\"ostrand \cite{MenSj}, O.~Matte \cite{Ma}, extending the approach of \cite{MeSj76} to non-homogeneous cases.) Finally it turned out to be more convenient to use Bargmann-FBI transforms in the spirit of \cite{Sj82} and \cite{HeSj86}. The semigroup method led to a strengthened result in case (B): The resolvent can be extended to a disc of radius ${\cal O}((h\ln 1/h)^{k/(k+1)})$ around the appropriate boundary points. This improvement has been obtained recently by W.~Bordeaux Montrieux \cite{Bo} for the model operator $hD_x+g(x)$, when $g\in C^\infty (S^1)$ and the points of maximum or minimum are all nondegenerate. In that case $k=2$ and Bordeaux Montrieux also constructed quasi-modes for values of the spectral parameter that are close to the boundary points. 

\par We next state the results and outline the proof in case (B).

\par Let $X$ be equal to ${\bf R}^n$ or equal to a compact smooth
manifold of dimension $n$. 

\par In the first case, let $m\in C^\infty ({\bf R}^{2n};[1,+\infty
[)$ be an order function (see \cite{DiSj99} for more details about the
pseudodifferential calculus) in the sense that for some $C_0,N_0>0$,
\ekv{in.1}
{
m(\rho )\le C_0\langle \rho -\mu \rangle^{N_0}m(\mu ),\ \rho ,\mu \in
{\bf R}^{2n},
}
where $\langle \rho -\mu \rangle= (1+|\rho -\mu |^2)^{1/2}$. Let
$P=P(x,\xi ;h)\in S(m)$, meaning that $P$ is smooth in $x,\xi $ and
satisfies 
\ekv{in.2}
{
|\partial _{x,\xi }^\alpha P(x,\xi ;h)|\le C_\alpha m(x,\xi ),\ (x,\xi
)\in {\bf R}^{2n},\, \alpha
\in {\bf N}^{2n},
}
where $C_\alpha $ is independent of $h$. We also assume that 
\ekv{in.3}
{
P(x,\xi ;h)\sim p_0(x,\xi )+hp_1(x,\xi )+...,\hbox{ in }S(m),
}
and write $p=p_0$ for the principal symbol. We impose the ellipticity
assumption
\ekv{in.4}
{
\exists w\in {\bf C},\, C>0,\hbox{ such that }|p(\rho )-w|\ge m(\rho )/C,\
\forall \rho \in {\bf R}^{2n}.
}
In this case we let 
\ekv{in.5}
{
P=P^w(x,hD_x;h)=\mathrm{Op}(P(x,h\xi ;h))
}
be the Weyl quantization of the symbol $P(x,h\xi ;h)$ that we can view as a
closed unbounded operator on $L^2({\bf R}^n)$.

\par In the second case when $X$ is compact manifold, we let $P\in
S^m_{1,0}(T^*X)$ (the classical H\"ormander symbol space ) of order
$m> 0$, meaning that
\ekv{in.6}
{
|\partial _x^\alpha \partial _{\xi }^\beta P(x,\xi ;h)|\le C_{\alpha
  ,\beta }\langle \xi \rangle^{m-|\beta |},\ (x,\xi )\in T^*X, } where $C_{\alpha ,\beta }$ are independent of $h$. We also assume that we have an expansion as in  (\ref{in.3}), now in the sense that \ekv{in.7} { P(x,\xi ;h)-\sum_0^{N-1}h^jp_j(x,\xi )\in h^NS^{m-N}_{1,0}(T^*X),\ N=1,2,...  } and we quantize the symbol $P(x,h\xi ;h)$ in the standard (non-unique) way, by doing it for various local coordinates and paste the quantizations together by means of a partition of unity. In the case $m>0$ we impose the ellipticity condition 
\ekv{in.8} { \exists C>0,\hbox{ such that }|p(x,\xi )|\ge \frac{\langle
  \xi \rangle^m}{C},\ |\xi |\ge C.
}

\par Let $\Sigma (p)=\overline{p^*(T^*X)}$ and let $\Sigma _\infty
(p)$ be the set of accumulation points of $p(\rho _j)$ for all
sequences $\rho _j\in T^*X$, $j=1,2,3,..$ that tend to infinity. The following
theorem is a partial improvement of corresponding results in
\cite{DeSjZw}.
\begin{theo}\label{in1}
We adopt the general assumptions above. Let $z_0\in \partial \Sigma
(p)\setminus \Sigma _\infty (p)$ and assume that $dp\ne 0$ at every
point of $p^{-1}(z_0)$. Then for every such point $\rho $ there exists
$\theta \in {\bf R}$ (unique up to a multiple of $\pi $) such that
$d(e^{-i\theta }(p-z_0))$ is real at $\rho $. We write $\theta
=\theta (\rho )$. Consider the following two cases:
\begin{itemize}
\item (A) For every $\rho \in p^{-1}(z_0)$, the maximal integral curve
  of $H_{\Re (e^{-i\theta (\rho )}p)}$ through the point $\rho $ is not
  contained in $p^{-1}(z_0)$. 
\item (B) There exists an integer $k\ge 1$ such that for every $\rho
  \in p^{-1}(z_0)$, there exists $j\in \{ 1,2,..,k\}$ such that 
$$p^*(\exp tH_p(\rho ))=at^j+{\cal O}(t^{j+1}),\ t\to 0,$$
where $a=a(\rho )\ne 0$. Here $p$ also denotes an almost holomorphic extesnion to a complex neighborhood of $\rho $ and we put $p^*(\mu )=\overline{p(\overline{\mu })}$. Equivalently, $H_p^j(\overline{p})(\rho )/(j!)=a\ne 0$.
\end{itemize}

\par Then, in case (A), there exists a constant $C_0>0$ such that for
every constant $C_1>0$ there is a constant $C_2>0$ such that the
resolvent $(z-P)^{-1}$ is well-defined for $|z-z_0|<C_1h\ln
\frac{1}{h}$, $h<\frac{1}{C_2}$, and satisfies the estimate 
\ekv{in.9}
{
\| (z-P)^{-1}\|\le \frac{C_0}{h}\exp ( \frac{C_0}{h}|z-z_0|).
}

\par In case (B), there exists a constant $C_0>0$ such that for
every constant $C_1>0$ there is a constant $C_2>0$ such that the
resolvent $(z-P)^{-1}$ is well-defined for $|z-z_0|<C_1(h\ln
\frac{1}{h})^{k/(k+1)}$, $h<\frac{1}{C_2}$ and satisfies the estimate 
\ekv{in.10}
{
\| (z-P)^{-1}\|\le \frac{C_0}{h^{\frac{k}{k+1}}}
\exp (\frac{C_0}{h}|z-z_0|^{\frac{k+1}{k}}).
}
\end{theo} 

\par In \cite{DeSjZw} we obtained (\ref{in.9}), (\ref{in.10}) for
$z=z_0$, implying that the resolvent exists and satisfies the same
bound for $|z-z_0|\le h^{k/(k+1)}/{\cal O}(1)$ in case (B) and with
$k/(k+1)$ replaced by 1 in case (A). In case (A) we also showed that
the resolvent  exists with norm bounded by a negative power of $h$
in any disc $D(z_0,C_1h\ln (1/h))$. (The condition in case (B) was formulated a little differently in \cite{DeSjZw}, but as we shall see later on the two conditions leed to the same microlocal models and hence they are equivalent.) Actually the proof in
\cite{DeSjZw} also gives (\ref{in.9}), so even if the methods of the
present paper also most likely lead to that bound, we shall not not elaborate the
details in that case.

Let us now consider the special situation of potential interest for
evolution equations, namely the case when
\ekv{in.11}{z_0\in i{\bf R},}
\ekv{in.12}{
\Re p(\rho )\ge 0\hbox{ in }\mathrm{neigh\,}(p^{-1}(z_0),T^*X).
}
\begin{theo}\label{in2}
We adopt the general assumptions above. Let $z_0\in \partial \Sigma
(p)\setminus \Sigma _\infty (p)$ and assume (\ref{in.11}),
(\ref{in.12}). Also assume that
$dp\ne 0$ on $p^{-1}(z_0)$, so that $d\Im p\ne 0$, $d\Re p=0$ on that set.
 Consider the two cases of Theorem \ref{in1}:
\begin{itemize}
\item (A) For every $\rho \in p^{-1}(z_0)$, the maximal integral curve
  of $H_{\Im p}$ through the point $\rho $ contains a point where $\Re
  p>0$. 
\item (B) There exists an integer $k\ge 1$ such that for every $\rho
  \in p^{-1}(z_0)$, we have $H_{\Im p}^j\Re p(\rho )\ne 0$ for some
  $j\in \{ 1,2,...,k\}$.
\end{itemize}

\par Then, in case (A), there exists a constant $C_0>0$ such that for
every constant $C_1>0$ there is a constant $C_2>0$ such that the
resolvent $(z-P)^{-1}$ is well-defined for 
$$
|\Im (z-z_0)|<\frac{1}{C_0},\ \frac{-1}{C_0}<\Re
z<C_1h\ln\frac{1}{h},\ h<\frac{1}{C_2},
$$
 and satisfies the estimate 
\ekv{in.13}
{
\| (z-P)^{-1}\|\le \cases{ \frac{C_0}{|\Re z|},\ \Re z\le -h,\cr
  \frac{C_0}{h}\exp (\frac{C_0}{h}\Re z), \Re z\ge -h.
 }
}

\par In case (B), there exists a constant $C_0>0$ such that for
every constant $C_1>0$ there is a constant $C_2>0$ such that the
resolvent $(z-P)^{-1}$ is well-defined for 
\ekv{in.13.5}
{|\Im (z-z_0)|<\frac{1}{C_0},\ \frac{-1}{C_0}<\Re
z<C_1(h\ln\frac{1}{h})^{\frac{k}{k+1}},\ h<\frac{1}{C_2},}
 and satisfies the estimate 
\ekv{in.14}
{
\| (z-P)^{-1}\|\le \cases{ \frac{C_0}{|\Re z|},\ \Re z\le -h^{\frac{k}{k+1}},\cr
  \frac{C_0}{h^{\frac{k}{k+1}}}\exp (\frac{C_0}{h}(\Re z)_+^{^{\frac{k}{k+1}}}), \Re z\ge -h^{\frac{k}{k+1}}.
 }
}
\end{theo} 

The case (A) in the theorems is practically identical with the corresponding results in \cite{DeSjZw} and can be obtained by inspection of the proof there, and from now on we concentrate on the case (B). 
Away from the set $p^{-1}(z_0)$ we can use ellipticity, so the problem is to obtain microlocal estimates near a point $\rho \in p^{-1}(z_0)$. After a standard factorization of $P-z$ in such a region, we can further reduce the proof of the first theorem to that of the second one.

\par The main (quite standard) idea of the proof of Theorem \ref{in2} is to study $\exp (-tP/h)$ (microlocally) for $0\le t\ll 1$ and to show that in this case
\ekv{in.15}
{
\| \exp -\frac{tP}{h}\| \le C\exp (-\frac{t^{k+1}}{Ch}),
}
for some constant $C>0$. Noting that that implies that $\| \exp -\frac{tP}{h}\| ={\cal O}(h^\infty )$ for $t\ge h^\delta $ when $\delta (k+1)<1$, and using the formula
\ekv{in.16}
{
(z-P)^{-1}=-\frac{1}{h}\int_0^\infty \exp (\frac{t(z-P)}{h}) dt,
}
leads to (\ref{in.14}). (This has some relation to the works of A.~Cialdea and Maz'ya \cite{CiMa05, CiMa06} where the $L^p$ dissipativity of second order operators is characterized.) 

The most direct way of studying $\exp (-tP/h)$, or rather a microlocal version of that operator, is to view it as a Fourier integral operator with complex phase (\cite{Mas, Ku, MeSj76, Ma}) of the form
\ekv{in.17}
{
U(t)u(x)=\frac{1}{(2\pi h)^n}\iint e^{\frac{i}{h}(\phi (t,x,\eta )-y\cdot \eta )}a(t,x,\eta ;h) u(y)dy d\eta ,
}
where the phase $\phi $ should have a non-negative imaginary part and satisfy the Hamilton-Jacobi equation:
\ekv{in.18}
{
i\partial _t\phi +p(x,\partial _x\phi )={\cal O}((\Im \phi )^\infty ), \hbox{ locally uniformly,}
}
with the initial condition 
\ekv{in.19}
{
\phi (0,x,\eta )=x\cdot \eta .
}
The amplitude $a$ will be bounded with all its derivatives and has an asymptotic expansion where the terms are determined by transport equations. This can indeed be carried out in a classical manner for instance by adapting the method of \cite{MeSj76} to the case of non-homogene\-ous symbols following a reduction used in \cite{MenSj, Ma}. It is based on making estimates on the fonction 
$$
S_\gamma (t)=\Im (\int_0^t \xi (s)\cdot dx(s))-\Re \xi (t)\cdot \Im x(t)+\Re \xi (0)\cdot \Im x(0)
$$ 
along the complex integral curves $\gamma :[0,T]\ni s\mapsto (x(s),\xi (s))$ of the Hamilton field of $p$. Notice that here and already in (\ref{in.18}), we need to take an almost holomorphic extension of $p$. Using the property (B) one can show that $\Im \phi (t,x,\eta )\ge C^{-1}t^{k+1}$ and from that we can obtain (a microlocalized version of) (\ref{in.15}) quite easily. 

Finally, we prefered a variant that we shall now outline: Let 
$$
Tu(x)=Ch^{-\frac{3n}{4}}\int e^{\frac{i}{h}\phi (x,y)}u(y)dy,
$$
be an FBI -- or (generalized) Bargmann-Segal transform that we treat in the spirit of Fourier integral operators with complex phase as in 
\cite{Sj82}. Here $\phi $ is holomorphic in a neighborhood of $(x_0,y_0)\in {\bf C}^n\times {\bf R}^n$, and $-\phi '_y(x_0,y_0)=\eta _0\in {\bf R}^n$, $\Im \phi ''_{y,y}(x_0,y_0)>0$, $\det \phi ''_{x,y}(x_0,y_0)\ne 0$. Let $\kappa _t$ be the associated canonical transformation. Then microlocally, $T$ is bounded $L^2\to H_{\Phi _0}:=\mathrm{Hol\,}(\Omega )\cap L^2(\Omega ,e^{-2\Phi _0/h}L(dx))$ and has (microlocally) a bounded inverse, where $\Omega $ is a small complex neighborhood of $x_0$ in ${\bf C}^n$. Here the weight $\Phi _0$ is smooth and strictly pluri-subharmonic. If $\Lambda _{\Phi _0}:=\{ (x,\frac{2}{i}\frac{\partial \Phi _0}{\partial x});\, x\in \mathrm{neigh\,}(x_0)\}$, then (in the sense of germs) $\Lambda _{\Phi _0}=\kappa _T(T^*X)$. The conjugated operator $\widetilde{P}=TPT^{-1}$ can be defined locally modulo ${\cal O}(h^\infty )$ (see also \cite{LaSj}) as a bounded operator from $H_\Phi \to H_\Phi $ provided that the weight $\Phi $ is smooth and satisfies $\Phi '-\Phi _0'={\cal O}(h^\delta) $ for some $\delta >0$. (In the analytic frame work this condition can be relaxed.) Egorov's theorem applies in this situation, so the leading symbol $\widetilde{p}$ of $\widetilde{P}$ is given by $\widetilde{p}\circ \kappa _T=p$. Thus (under the assumptions of Theorem \ref{in2}) we have ${{\Re \widetilde{p}}_\vert}_{\Lambda _{\Phi _0}}\ge 0$, which in turn can be used to see that for $0\le t\le h^\delta $, we have $e^{-t\widetilde{P}/h}={\cal O}(1)$: 
$H_{\Phi _0}\to H_{\Phi _t}$, where $\Phi _t\le \Phi _0$ is determined by the real Hamilton-Jacobi problem
\ekv{in.20}
{
\frac{\partial \Phi _t}{\partial t}+\Re \widetilde{p}(x,\frac{2}{i}\frac{\partial \Phi _t}{\partial x})=0,\ \Phi _{t=0}=\Phi _0.
}
Now the bound (\ref{in.15}) follows from the estimate
\ekv{in.21}
{
\Phi _t\le \Phi _0-\frac{t^{k+1}}{C}
}
where $C>0$. An easy proof of (\ref{in.21}) is to represent the I-Lagrangian manifold $\Lambda _{\Phi _t}$ as the image under $\kappa _T$ of the I-Lagrangian manifold $\Lambda _{G_t}=\{ \rho +iH_{G_t}(\rho );\,\rho \in \mathrm{neigh\,}(\rho _0,T^*X)\}$, where $H_{G_t}$ denotes the Hamilton field of $G_t$. It turns out that the $G_t$ are given by the real Hamilton-Jacobi problem
\ekv{in.22}
{
\frac{\partial G_t}{\partial t}+\Re (p(\rho +iH_{G_t}(\rho )))=0,\ G_0=0,
}
and there is a simple minimax type formula expressing $\Phi _t$ in terms of $G_t$, so it suffices to show that 
\ekv{in.23}
{
G_t\le -t^{k+1}/C.
}

This estimate is quite simple to obtain: (\ref{in.22}) first implies that $G_t\le 0$, so $(\nabla G_t)^2={\cal O}(G_t)$. Then if we Taylor expand (\ref{in.22}), we get
$$
\frac{\partial G_t}{\partial t}+H_{\Im p}(G_t)+{\cal O}(G_t)+\Re p(\rho )=0
$$
and we obtain (\ref{in.23}) from a simple differential inequality and an estimate for certain integrals of $\Re p$.

\par The use of the representation with $G_t$ is here very much taken from the joint work \cite{HeSj86} with B.~Helffer.

\par In Section \ref{ex} we discuss some examples.

\section{IR-manifolds close to ${\bf R}^{2n}$ and their FBI-representations}\label{ge}
\setcounter{equation}{0}

Much of this section is just an adaptation of the discussion in
\cite{HeSj86} with the difference that we here use the simple
FBI-transforms of generalized Bargmann type from \cite{Sj82}, rather
than the more complicated variant that was necessary to treat a
neighborhood of infinity in the resonance theory of \cite{HeSj86}.

We shall work locally. Let $G(y,\eta )\in C^\infty
(\mathrm{neigh\,}((y_0,\eta _0),{\bf R}^{2n}))$ be real-valued and
small in the $C^{\infty }$ topology. Then 
$$
\Lambda _G=\{ (y,\eta )+iH_G(y,\eta );\, (y,\eta )\in \mathrm{neigh\,}((y_0,\eta _0))\},\ H_G=\frac{\partial G}{\partial \eta }\frac{\partial }{\partial y
}-
\frac{\partial G}{\partial y}\frac{\partial }{\partial \eta }
$$
is an $I$-Lagrangian manifold, i.e.~a Lagrangian manifold for the real
symplectic form $\Im \sigma $, where $\sigma $ denotes the complex
symplectic form $\sum_1^n d\widetilde{\eta }_j\wedge
d\widetilde{y}_j$. Here, for notational reasons we reserve the
notation $(y,\eta )$ for the real cotangent variables and let the
tilde indicate that we take the corresponding complexified variables.

We may also represent $\Lambda _G$ by means of a nondegenerate phase
function in the sense of H\"ormander in the following way: 

Consider 
$$
\psi (\widetilde{y},\eta )=-\eta \cdot \Im \widetilde{y}+G(\Re
\widetilde{y},\eta )
$$
where $\widetilde{y}$ is complex and $\eta $ real according to the
convention above. Then 
$$
\nabla _\eta \psi (\widetilde{y},\eta )=-\Im \widetilde{y}+\nabla
_\eta G(\Re \widetilde{y},\eta ),
$$  
and since $G$ is small, we see that $d\frac{\partial \psi }{\partial
  \eta _1},...,d\frac{\partial \psi }{\partial \eta _n}$ are linearly
independent. So $\psi $ is indeed a nondegenerate phase function if we drop the classical requirement of
homogeneity in the $\eta $ variables.

Let 
$$
C_\psi =\{ (\widetilde{y},\eta )\in \mathrm{neigh\,}((y_0,\eta _0),{\bf C}^n\times {\bf R}^n);\, \nabla _\eta \psi =0\}
$$ and consider the corresponding I-Lagrangian manifold
$$
\Lambda _\psi =\{ (\widetilde{y},\frac{2}{i}\frac{\partial \psi
}{\partial \widetilde{y}}(\widetilde{y},\eta ));\, (\widetilde{y},\eta
)\in C_\psi \} .
$$
Here we adopt the convention that $\frac{\partial }{\partial
  \widetilde{y}}$ denotes the holomorphic derivative, since
$\widetilde{y}$ are complex variables:
$$
\frac{\partial }{\partial \widetilde{y}}=\frac{1}{2}(\frac{\partial
}{\Re \widetilde{y}}+\frac{1}{i}\frac{\partial }{\partial \Im \widetilde{y}}).
$$
Let us first check that that $\Lambda _\psi $ is I-Lagrangian, using
only that $\psi $ is a nondegenerate phase function: That
$\Lambda _\psi $ is a submanifold with the correct real dimension $=2n$
is classical since we can identify $\frac{2}{i}\frac{\partial \psi
}{\partial \widetilde{y}}$ with $\nabla _{\Re \widetilde{y}, \Im
  \widetilde{y}}\psi $. Further,
\begin{eqnarray*}
&&-\Im {{(\widetilde{\eta }\cdot d\widetilde{y})}_\vert}_{\Lambda
  _\psi }\simeq -\Im {{(\frac{2}{i}\frac{\partial \psi }{\partial
  \widetilde{y}}\cdot d\widetilde{y})}_\vert}_{C_{\psi }}=\\
&&-\frac{1}{2i}{{(\frac{2}{i}\frac{\partial \psi }{\partial
      \widetilde{y}}d\widetilde{y}+\frac{2}{i}\frac{\partial \psi
    }{\partial
      \overline{\widetilde{y}}}d\overline{\widetilde{y}})}_\vert}_{C_\psi }=
{{(\frac{\partial \psi }{\partial
      \widetilde{y}}d\widetilde{y}+\frac{\partial \psi }{\partial
      \overline{\widetilde{y}}}d\overline{\widetilde{y}})}_\vert}_{C_\psi}\\
&&={{d\psi }_\vert}_{C_\psi }
\end{eqnarray*}
which is a closed form and using that $\Im \sigma =d\Im (\widetilde{\eta
}\cdot d\widetilde{y})$, we get 
$$
-{{\Im \sigma }_\vert}_{\Lambda _\psi }=0.
$$

We next check for our specific phase $\psi $ that $\Lambda _{\psi
  }=\Lambda _G$: If $(\widetilde{y},\frac{2}{i}\frac{\partial \psi
  }{\partial \widetilde{y}}(\widetilde{y},\eta ))$ is a general point
  on $\Lambda _\psi $, then $\Im \widetilde{y}=\nabla _\eta G(\Re
  \widetilde{y},\eta )$ and
\begin{eqnarray*}
\frac{2}{i}\frac{\partial \psi }{\partial
  \widetilde{y}}(\widetilde{y},\eta )&=&\frac{2}{i}\frac{1}{2}
(\frac{\partial }{\partial \Re
  \widetilde{y}}+\frac{1}{i}\frac{\partial }{\partial \Im
  \widetilde{y}})
(-\eta \cdot \Im \widetilde{y} +G(\Re \widetilde{y},\eta ))\\
&=& -(\frac{\partial }{\partial \Im \widetilde{y}}+i\frac{\partial
}{\partial \Re \widetilde{y}})(-\eta \cdot \Im \widetilde{y} +G(\Re
\widetilde{y},\eta ))\\
&=& \eta -i\nabla _y G(\Re \widetilde{y},\eta ).
\end{eqnarray*}
Hence 
$$
(\widetilde{y},\frac{2}{i}\frac{\partial \psi }{\partial
  \widetilde{\eta }})
=(y,\eta )+iH_G(y,\eta ),
$$
if we choose $y=\Re \widetilde{y}$.
\hfill{$\Box$}

\medskip Now consider an FBI (or generalized Bargmann-Segal) transform
$$
Tu(x;h)=h^{-\frac{3n}{4}}\int e^{i\phi (x,y)/h}a(x,y;h)u(y)u(y)dy,
$$
where $\phi $ is holomorphic near $(x_0,y_0)\in {\bf C}^n\times {\bf
  R}^n$, $\Im \phi ''_{y,y}>0$, $\det \phi ''_{x,y}\ne 0$,
$-\frac{\partial \phi }{\partial y}=\eta _0\in {\bf R}^n$, and $a$ is holomorphic in the same neighborhood with $a\sim a_0(x,y)+ha_1(x,y)+...$ in the space of such functions with $a_0\ne 0$. We can view
$T$ as a Fourier integral operator with complex phase and the
associated canonical transformation is
$$
\kappa =\kappa _T:\ (y,-\frac{\partial  \phi }{\partial y}(x,y))\mapsto
(x,\frac{\partial \phi }{\partial x}(x,y))
$$
from a complex neighborhood of $(y_0,\eta _0)$ to a complex
neighborhood of $(x_0,\xi _0)$, where $\xi _0=\frac{\partial \phi
}{\partial x}(x_0,y_0)$. Complex canonical transformations preserve
the class of I-Lagrangian manifolds and (locally), 
$$
\kappa ({\bf R}^{2n})=\Lambda _{\Phi _0}=\{
(x,\frac{2}{i}\frac{\partial \Phi _0}{\partial x}(x));\ x\in
\mathrm{neigh\,}(x_0,{\bf C}^n)\},
$$
where $\Phi _0$ is smooth and strictly plurisubharmonic. Actually,
\ekv{ir.1}
{
\Phi _0(x)=\sup_{y\in {\bf R}^n}-\Im \phi (x,y),
}
where the supremum is attained at the nondegenerate point of maximum $y_c(x)$ (\cite{Sj82}).
\begin{prop}\label{ir1}
We have $\kappa (\Lambda _G)=\Lambda _{\Phi _G}$, where
\ekv{ir.2}
{
\Phi _G(x)=\mathrm{v.c.}_{\widetilde{y},\eta }-\Im \phi
(x,\widetilde{y})-\eta \cdot \Im \widetilde{y}+G(\Re
\widetilde{y},\eta ),
}
and the critical value is attained at a nondegenerate critical point.
Here $\mathrm{v.c.}_{\widetilde{y},\eta }(...)$ means ``critical value
with respect to $\widetilde{y},\eta $ of ...''. 
\end{prop}
\begin{proof}
At a critical point we have
\begin{eqnarray*}
&\Im \widetilde{y}=\nabla _\eta G(\Re \widetilde{y},\eta ),&\\
&\frac{\partial }{\partial \Im \widetilde{y}}\Im \phi
(x,\widetilde{y})+\eta =0,&\\
&-\frac{\partial }{\partial \Re \widetilde{y}}\Im \phi
(x,\widetilde{y})+(\nabla _y G)(\Re \widetilde{y},\eta )=0.&
\end{eqnarray*}
If $f(z)$ is a holomorphic function, then 
\ekv{ir.3}{\frac{\partial }{\partial \Im z}\Im f=\Re 
\frac{\partial f}{\partial z}
,\ \frac{\partial }{\partial \Re z}\Im f=
\Im \frac{\partial f}{\partial z},}
so the equations for our critical point become
\begin{eqnarray*}
\Im \widetilde{y}&=&\nabla _\eta G(\Re \widetilde{y},\eta ),\\
\eta &=&-\Re \frac{\partial \phi }{\partial \widetilde{y}}(x,\widetilde{y}),\\
-\nabla _yG(\Re \widetilde{y},\eta )&=&-\Im \frac{\partial \phi
}{\partial \widetilde{y}},
\end{eqnarray*}
or equivalently,
$$
(\widetilde{y},-\frac{\partial \phi }{\partial
  \widetilde{y}}(x,\widetilde{y}))=(\Re \widetilde{y},\eta )+iH_G(\Re
\widetilde{y},\eta ),
$$  
which says that the critical point $(\widetilde{y},\eta )$ is
determined by the condition that $\kappa _T$ maps the point
$(\widetilde{y},\widetilde{\eta })\in \Lambda _G$ to a point $(x,\xi
)$, situated over $x$. Clearly the critical point is nondegenerate. We
check it when $G=0$: The Hessian matrix with respect to the variables
$\Re y,\Im y,\eta $ becomes
$$
\left(\begin{array}{ccc}-\Im \phi ''_{y,y} &B &0\\
\trans{B} &C &-1\\
0 &-1 &0
 \end{array}\right)
$$
which is nondegenerate independently of $B,C$.

\par If $\Phi (x)$ denotes the critical value in (\ref{ir.2}), it
remains to check that $\frac{2}{i}\frac{\partial \Phi }{\partial
  x}=\xi $ where $\xi =\frac{\partial \phi }{\partial
  x}(x,\widetilde{y})$, $(\widetilde{y},\eta )$ denoting the critical
point. However, since $\Phi $ is a critical value, we get
$$
\frac{2}{i}\frac{\partial \Phi }{\partial x}=\frac{2}{i}\frac{\partial
}{\partial x}(-\Im \phi (x,\widetilde{y}))=\frac{\partial \phi
}{\partial x}(x,\widetilde{y}).
$$
\end{proof}

\par Also notice that when $G=0$, the formula (\ref{ir.2}) produces
the same function as (\ref{ir.1}).

\par Write $\widetilde{y}=y+i\theta $ and consider the function
\ekv{ir.5}
{
f(x;y,\eta ;\theta )=-\Im \phi (x,y+i\theta )-\eta \cdot \theta ,
}
which appears in (\ref{ir.2}).
\begin{prop}\label{ir2}
$f$ is a nondegenerate phase function with $\theta $ as fiber
variables which generates a canonical transformation which can be
identified with $\kappa _T$.
\end{prop}
\begin{proof}
$$
\frac{\partial }{\partial \theta }f=-\Re \frac{\partial \phi
}{\partial \widetilde{y}}(x,y+i\theta )-\eta ,
$$
so $f$ is nondegenerate. The canonical relation has the graph
\begin{eqnarray*}
&\{ (x,\frac{\partial \phi }{\partial x};y,\eta ,\frac{\partial
}{\partial y}\Im \phi (x,y+i\theta ),\theta );\ \eta =-\Re
\frac{\partial \phi }{\partial \widetilde{y}}(x,y+i\theta )\}&\\
&=\{ (x,\frac{\partial \phi }{\partial x}(x,y+i\theta );y,-\Re
\frac{\partial \phi }{\partial \widetilde{y}}(x,y+i\theta ),\Im
\frac{\partial \phi }{\partial \widetilde{y}}(x,y+i\theta ),\theta )\}&,
\end{eqnarray*}
and up to reshuffling of the components on the preimage side and changes of signe, we
recognize the graph of $\kappa _T$.
\end{proof}

\par Now we have the following easily verified fact:
\begin{prop}\label{ir3}
Let $f(x,y,\theta )\in C^\infty (\mathrm{neigh\,}(x_0,y_0,\theta
_0),{\bf R}^n\times {\bf R}^n\times {\bf R}^N)$ be a nondegenerate phase
function with $(x_0,y_0,\theta _0)\in C_\phi $, generating a canonical
transformation which maps $(y_0,\eta _0)=(y_0,-\nabla
_yf(x_0,y_0,\theta _0))$ to $(x_0,\nabla _xf(x_0,y_0,\theta _0))$. If
$g(y)$ is smooth near $y_0$ with $\nabla g(y_0)=\eta _0$ and
$$h(x)=\mathrm{v.c.}_{y,\theta }f(x,y,\theta )+g(y)$$ is well-defined
with a nondegenerate critical point close to $(y_0,\theta _0)$ for
$x$ close to $x_0$, then we have the inversion formula,
$$
g(y)=\mathrm{v.c.}_{x,\theta }-f(x,y,\theta )+h(x),
$$ 
for $y\in \mathrm{neigh\,}(y_0)$, where the critical point is
nondegenerate and close to $(x_0,\theta _0)$.
\end{prop}

\par Combining the three propositions, we get
\begin{prop}\label{ir4}
\ekv{ir.6}
{
G(y,\eta )=\mathrm{v.c.}_{x,\theta }\Im \phi (x,y+i\theta )+\eta \cdot
\theta +\Phi _G(x).
}
\end{prop}

If $(\widetilde{\Phi
},\widetilde{G})$ is a second pair of functions close to $\Phi _0$,
$0$ and related through (\ref{ir.2}), (\ref{ir.6}), then
\ekv{ir.7}{G\le \widetilde{G}\hbox{ iff }\Phi \le \widetilde{\Phi }.}
Indeed, if for instance $\Phi \le \widetilde{\Phi }$, introduce $\Phi _t=t\widetilde{\Phi }+(1-t)\Phi $, so that $\partial _t\Phi _t\ge 0$. If $G_t$ is the corresponding critical value as in (\ref{ir.6}), then 
$\partial _tG_t=(\partial _t\Phi _t)(x_t)\ge 0$, where $(x_t,\theta _t)$ is the critical point.
\section{Evolution equations on the transform side}\label{ev}
\setcounter{equation}{0}

Let $\widetilde{P}(x,\xi ;h)$ be a smooth symbol defined in
$\mathrm{neigh\,}((x_0,\xi _0);\Lambda _{\Phi _0})$, with an
asymptotic expansion
$$\widetilde{P}(x,\xi ;h)
\sim \widetilde{p}(x,\xi )+h\widetilde{p}_1(x,\xi )+...\hbox{ in
}C^\infty (\mathrm{neigh\,}((x_0,\xi _0),\Lambda _{\Phi _0})).
$$
By the same letter, we denote an almost holomorphic extension to a
complex neighborhood of $(x_0,\xi _0)$:
$$\widetilde{P}(x,\xi ;h)
\sim \widetilde{p}(x,\xi )+h\widetilde{p}_1(x,\xi )+...\hbox{ in
}C^\infty (\mathrm{neigh\,}((x_0,\xi _0),{\bf C}^{2n}),
$$
where $\widetilde{p}$, $\widetilde{p}_j$ are smooth
extensions such that 
$$\overline{\partial }\widetilde{p},\ \overline{\partial
}\widetilde{p}_j={\cal O}(\mathrm{dist\,}((x,\xi ),\Lambda _{\Phi
  _0})^\infty). 
$$
Then, as developed in \cite{LaSj} and later in \cite{MeSj02}, if $u=u_h$ is holomorphic in a
neighborhood $V$ of $x_0$ and belonging to $H_{\Phi _0}(V)$ in the
sense that $\Vert u\Vert_{L^2(V,e^{-2\Phi _0/h}L(dx))}$ is finite and
of temperate growth in $1/h$ when $h$ tends to zero, then we can
define $\widetilde{P}u=\widetilde{P}(x,hD_x;h)u$ in any smaller
neighborhood $W\Subset V$ by the formula,
\ekv{ev.0.1}
{
\widetilde{P}u(x)=\frac{1}{(2\pi h)^n}\iint_{\Gamma
  (x)}e^{\frac{i}{h}(x-y)\cdot \theta
}\widetilde{P}(\frac{x+y}{2},\theta ;h)u(y)dyd\theta ,
}
where $\Gamma (x)$ is a good contour (in the sense of \cite{Sj82}) of
the form $\theta =\frac{2}{i}\frac{\partial \Phi _0}{\partial
  x}(\frac{x+y}{2})+\frac{i}{C_1}(\overline{x-y})$, $
|x-y|\le 1/C_2$, $C_1,C_2>0$. Then $\overline{\partial }\widetilde{P}$ is negligible
$H_{\Phi _0}(V)\to L^2_{\Phi _0}(W)$, i.e. of norm ${\cal O}(h^\infty
)$ and modulo such negligible operators, $\widetilde{P}$ is
independent of the choice of good contour. By solving a
$\overline{\partial }$-problem (assuming, as we may, that our neighborhoods are pseudoconvex) we can always correct $\widetilde{P}$
with a negligible operator such that (after an arbitrarily small
decrease of $W$) $\widetilde{P}={\cal O}(1):H_{\Phi _0}(V)\to H_{\Phi
  _0}(W)$.
Also, if $\Phi =\Phi _0+{\cal O}(h\ln \frac{1}{h})$ in $C^2$, then
clearly $\widetilde{P}={\cal O}(h^{-N_0}): H_{\Phi }(V)\to H_{\Phi
}(W)$, for some $N_0$. Using Stokes' formula, we can show
that $\widetilde{P}$ will change only by a negligible term if we
replace $\Phi _0$ by $\Phi $ in the definition of $\Gamma (x)$, and
then it follows that $\widetilde{P}={\cal O}(1):H_{\Phi }(V)\to 
H_{\Phi}(W)$.

Before discussing evolution equations, let us recall (\cite{MeSj02})
that the identity operator $H_{\Phi_0} (V)\to H_{\Phi_0}
 (W)$ is up to a
negligible operator of the form
\ekv{ev.0.2}
{
Iu(x)=h^{-n}\iint e^{\frac{2}{h}\Psi _0(x,\overline{y})}a(x,\overline{y};h)u(y)e^{-\frac{2}{h}\Phi _0(y)}dyd\overline{y},
}
where $\Psi _0(x,y)$, $a(x,y;h)$ are almost holomorphic on the
antidiagonal $y=\overline{x}$ with $\Psi _0(x,\overline{x})=\Phi
_0(x)$, $a(x,y;h)\sim a_0(x,y)+ha_1(x,y)+...$, $a_0(x,\overline{x})\ne
0$. More generally a
pseudodifferential operator like $\widetilde{P}$ takes the form
\eekv{ev.0.3}
{
\widetilde{P}u(x)&=&h^{-n}\iint e^{\frac{2}{h}\Psi _0(x,\overline{y})}q
(x,\overline{y};h)u(y)e^{-\frac{2}{h}\Phi _0(y)}dyd\overline{y}
}
{q_0(x,\overline{x})&=&\widetilde{p}(x,\frac{2}{i}\frac{\partial \Phi _0}{\partial
    x}(x))a_0(x,\overline{x}),}
and where $q_0$ denotes the first term in the asymptotic expansion of
the symbol $q$. In this discussion, $\Phi _0$ can be replaced by any
other smooth exponent $\Phi$ which is ${\cal O}(h^\delta ) $ close to 
$\Phi_0 $ in $C^\infty $ and we make the corresponding replacement of
$\Psi _0$. Also recall that because of the strict pluri-subharmonicity
of $\Phi $, we have
\ekv{ev.0.4}
{
2\Re \Psi (x,\overline{y})-\Phi (x)-\Phi (y)\asymp -|x-y|^2,
}
so the uniform boundedness $H_\Phi \to H_\Phi $ follows from the domination of the modulus of the effective kernel by a Gaussian convolution kernel. 

\par Next, consider the evolution problem
\ekv{ev.0.5}
{
(h\partial _t+\widetilde{P})\widetilde{U}(t)=0,\ \widetilde{U}(0)=1,
}
where $t$ is restricted to the interval $[0,h^\delta ]$ for some
arbitrarily small but fixed $\delta >0$. We review how to solve this
problem approximately by a geometrical optics construction:
Look for $\widetilde{U}(t)$ of the form
\ekv{ev.0.6}
{
\widetilde{U}(t)u(x)=h^{-n}\iint e^{\frac{2}{h}\Psi
  _t(x,\overline{y})}
a_t(x,\overline{y};h) u(y) e^{-2\Phi _0(y)/h}dyd\overline{y},
}
where $\Psi _t$, $a_t$ depend smoothly on all the variables and
$\Psi _{t=0}=\Psi _0$, $a _{t=0}=a _0$ in (\ref{ev.0.3}), so that $\widetilde{U}(0)=1$
up to a negligible operator.

\par Notice that formally $\widetilde{U}(t)$ is the Fourier integral
operator
\ekv{ev.0.7}
{
\widetilde{U}(t)u(x)=h^{-n}\iint e^{\frac{2}{h}(\Psi _t(x,\theta
  )-\Psi _0(y,\theta ))}a_t(x,\theta ;h)u(y) dyd\theta ,
}
where we choose the integration contour $\theta =\overline{y}$.
Writing $2\Psi _t(x,\theta )=i\phi _t(x,\theta )$ leads to more
customary notation and we impose the eiconal equation
\ekv{ev.0.8}
{
i\partial _t\phi +\widetilde{p}(x,\phi '_x(x,\theta ))=0.
}
Of course, we are manipulating $C^\infty $ functions in the complex
domain, so we cannot hope to solve the eiconal equation exactly, but
we can do so to infinite order at $t=0$, $x=\overline{y}=\theta $. If
we put 
\ekv{ev.0.10}
{
\Lambda _{\phi _t(\cdot ,\theta )}=\{ (x,\phi '_x(t,x,\theta ))\},
}
we have to $\infty $ order at $t=0$, $\theta =x$:
\ekv{ev.0.11}
{
\Lambda _{\phi _t(\cdot ,\theta )}=\exp (t\widehat{H}_{\frac{1}{i}\widetilde{p}})
(\Lambda _{\phi _0(\cdot ,\theta )}).
}
with
$\widehat{H}_{\widetilde{p}}=H_{\widetilde{p}}+\overline{H_{\widetilde{p}}}$
denoting the real vector field associated to the (1,0)-field $H_p$, and similarly for $\widehat{H}_{\frac{1}{i}\widetilde{p}}$. (We sometimes neglect the hat when integrating the Hamilton flows.)
At a point where $\overline{\partial }\widetilde{p}=0$, we have
\ekv{ev.0.12}
{
\widehat{H}_{\widetilde{p}}=H^{\Re \sigma }_{\Re\widetilde{p}}=H^{\Im\sigma
}_{\Im\widetilde{p}},\quad \widehat{H}_{i\widetilde{p}}=
-H^{\Re\sigma }_{\Im\widetilde{p}}=H^{\Im \sigma }_{\Re\widetilde{p}},
}
where the other fields are the Hamilton fields of $\Re \widetilde{p}$,
$\Im \widetilde{p}$ with respect to the real symplectic forms $\Re
\sigma $ and $\Im \sigma $ respectively. See \cite{Sj82, MeSj02}. Thus (\ref{ev.0.11}) can be
written
\ekv{ev.0.13}
{
\Lambda _{\phi _t(\cdot ,\theta )}=\exp (tH^{-\Im \sigma }_{\Re
  \widetilde{p}})(\Lambda _{\phi _0(\cdot ,\theta )}).
}
A complex Lagrangian manifold is also an I-Lagrangian manifold (i.e. a
Lagrangian manifold for $\Im \sigma $) so (\ref{ev.0.13}) can be
viewed as a relation between I-Lagrangian manifolds and it defines the
I-Lagrangian manifold $\Lambda _{\phi _t(\cdot ,\theta )}$
 in an
unambigious way, once we have fixed an almost holomorphic extension of
$\widetilde{p}$ and especially the real part of that function.
The general form of a smooth I-Lagrangian manifold $\Lambda $, for
which the $x$-space projection $\Lambda\ni (x,\xi ) \mapsto x\in{\bf
  C}^n$ is a local diffeomorphism, is locally $\Lambda =\Lambda _\Phi $
where $\Phi $ is real and smooth and we define 
$$
\Lambda _\Phi = \{ (x,\frac{2}{i}\frac{\partial \Phi }{\partial x});\, x\in \Omega \} ,\ \Omega \subset {\bf C}^n\hbox{ open}.
$$
With a slight abuse of notation, we can therefore identify the ${\bf
  C}$-Lagrangian manifold $\Lambda _{\phi _0}$ with the I-Lagrangian
manifold $\Lambda _{-\Im \phi _0}$, since for holomorphic functions
(or more generally where $\overline{\partial }_x\phi _0=0$, we have 
$\frac{\partial \phi _0}{\partial x}=\frac{2}{i}\frac{\partial -\Im
  \phi _0}{\partial x}$.

(\ref{ev.0.4}) shows that 
$$\Phi _0(x)+\Phi _0(\overline{\theta } )-(-\Im \phi _0 (\cdot ,\overline{\theta} ))\asymp
|x-\overline{\theta }|^2.$$ Thus, if we define 
\ekv{ev.0.16}
{
\Lambda _{\Phi _t}=\exp (tH_{\Re \widetilde{p}}^{-\Im \sigma })(\Lambda
_{\Phi _0}),
}
and fix the $t$-dependent constant in this defintition of $\Phi _t$ by
imposing the real Hamilton-Jacobi equation,
\ekv{ev.0.17}
{
\partial _t\Phi _t+\Re \widetilde{p}(x,\frac{2}{i}\frac{\partial \Phi
  _t}{\partial x})=0,\ \Phi _{t=0}=\Phi _0,
} and noticing that the real part of (\ref{ev.0.8}) is a similar
equation for $-\Im \phi _t$,
\ekv{ev.0.18}
{
\partial _t(-\Im \phi
)+\Re \widetilde{p}(x,\frac{2}{i}\frac{\partial }{\partial
  x}(-\Im \phi ))=0,
}
we get 
\ekv{ev.0.19}
{
\Phi _t(x)+\Phi _0(\overline{\theta} )-(-\Im \phi _t(x,\theta ))\asymp |x-x_t(\overline{\theta} )|^2,
}
where $(x_t(\overline{\theta} ),\xi _t(\overline{\theta} )):=\exp (tH_{\Re
  \widetilde{p}}^{-\Im \sigma })(\overline{\theta} ,\frac{2}{i}\frac{\partial \Phi
_0}{\partial x}(\overline{\theta} ))$.

Determining $a_t$ by solving a sequence of transport equations, we
arrive at the following result:
\begin{prop}\label{ev01}
The operator $\widetilde{U}(t)$ constructed above is ${\cal
  O}(1):H_{\Phi _0}(V)\to H_{\Phi _t}(W)$, ($W\Subset V$ being small pseudoconvex neighborhoods of a fixed point $x_0$) uniformly for $0\le t\le h^\delta
$ and it solves the problem (\ref{ev.0.5}) up to negligible
terms. This local statement makes sense, since by (\ref{ev.0.19}) we
have
\ekv{ev.0.20}
{
2\Re \Psi _t(x,\overline{y})-\Phi _t(x)-\Phi _0(y)\asymp -|x-x_t(y)|^2.
}
\end{prop}
Using standard arguments, we also obtain up to negligible errors
\ekv{ev.0.21}
{
h\partial _t\widetilde{U}(t)+\widetilde{U}(t)\widetilde{P}=0,\ 0\le
t\le h^\delta .
}

Let us quickly outline an alternative approach leading to the same weights $\Phi _t$ (cf \cite{MeSj02}):

\par Consider formally:
$$
(e^{-t\widetilde{P}/h}u|e^{-t\widetilde{P}/h}u)_{H_{\Phi _t}}=(u_t|u_t)_{H_{\Phi _t}},\ u\in H_{\Phi _0},
$$
and try to choose $\Phi _t$ so that the time derivative of this expression
vanishes to leading order. We get
\begin{eqnarray*}0&\approx&h\partial _t\int u_t\overline{u}_te^{-2\Phi
    _t/h}L(dx)\\
&=&-\left( (\widetilde{P}u_t|u_t)_{H_{\Phi _t}}+
(u_t|\widetilde{P}u_t)_{H_{\Phi _t}}+\int 2\frac{\partial \Phi
  _t}{\partial t}(x)|u|^2 e^{-2\Phi _t/h}L(dx)
\right).
\end{eqnarray*}
Here 
$$
(\widetilde{P}u_t|u_t)_{H_{\Phi _t}}=\int
(\widetilde{p}_{\vert_
{\Lambda _{\Phi _t}}}+{\cal
  O}(h))|u_t|^2e^{-2\Phi _t/h}L(dx),
$$
and similarly for $(u_t|\widetilde{P}u_t)_{H_{\Phi _t}}$, so we would
like to have
$$
0\approx
\int (2\frac{\partial \Phi _t}{\partial t}+2\Re 
\widetilde{p}_{\vert_
{\Lambda _{\Phi _t}}}+{\cal O}(h))|u_t|^2
e^{-2\Phi _t/h}L(dx).
$$

\par We choose $\Phi _t$ to be the solution of (\ref{ev.0.17}). 
 Then the preceding
discussion again shows that $e^{-t\widetilde{P}/h}={\cal O}(1): H_{\Phi
  _0}\to H_{\Phi _t}$. 

\par
\medskip\par
Since $\Re
\widetilde{p}$ is constant along the integral curves of $H_{\Re
  \widetilde{p}}^{-\Im \sigma }$, we see from (\ref{ev.0.17}), that the
second term in (\ref{ev.0.17}) is $\ge 0$, so 
\ekv{ev.2.5}{\Phi _t\le \Phi _0,\ t\ge 0,}
when
\ekv{ev.0}{{\Re\widetilde{p}_\vert}_{\Lambda _{\Phi _0}}\ge 0.}
Recall that we limit our discussion to the interval $0\le t\le
h^\delta $.

The author found it simpler to get a detailed understanding by working
with the corresponding functions $G_t$ in the following way:

Let $p$ be defined by $p=\widetilde{p}\circ \kappa _T$ and define
$G_t$ up to a $t$-dependent constant by 
$$
\Lambda _{\Phi _t}=\kappa _T(\Lambda _{G_t}).
$$
Then we also have $\Lambda _{G_t}=\exp tH_{p}(\Lambda
_0)$, where $\Lambda _0={\bf R}^{2n}$. In order to fix the
$t$-dependent constant we use one of the equivalent formulae (cf
(\ref{ir.2}), (\ref{ir.6})):
\ekv{ev.3}
{
\Phi _t(x)=\mathrm{v.c.}_{\widetilde{y},\eta }(-\Im \phi
(x,\widetilde{y})-\eta \cdot \Im \widetilde{y}+G_t(\Re
\widetilde{y},\eta )),
} 
\ekv{ev.4}
{
G_t(y,\eta )=\mathrm{v.c.}_{x,\theta }(\Im \phi (x,y+i\theta ) +\eta
\cdot \theta +\Phi _t(x)).
}
If $(x(t,y,\eta ),\theta (t,y,\eta )) $ is the critical point in the
last formula, we get 
\ekv{ev.5}
{
\frac{\partial G_t}{\partial t}(y,\eta )=\frac{\partial \Phi
  _t}{\partial t}(x(t,y,\eta ))=-{{\Re \widetilde{p}(x,\frac{2}{i}\frac{\partial
  \Phi _t}{\partial x})}_\vert}_{x=x(t,y,\eta )}.
}
As we have seen, the critical points in (\ref{ev.3}), (\ref{ev.4}) are
directly related to $\kappa _T$, so (\ref{ev.5}) leads to 
\ekv{ev.6}{\frac{\partial G_t}{\partial t}(y,\eta )+\Re p((y,\eta
)+iH_{G_t}(y,\eta ))=0.}
Notice that $G_t\le 0$ by (\ref{ir.7}), (\ref{ev.2.5}).

\par Since we consider (\ref{ev.6}) only when $G_t$ and its gradient
are small, we can Taylor expand (\ref{ev.6}) and get
\ekv{ev.7}{
\frac{\partial G_t}{\partial t}(y,\eta )+\Re p(y,\eta )+\Re
(iH_{G_t}p(y,\eta ))+{\cal O}((\nabla G_t)^2)=0,
}
which simplifies to 
\ekv{ev.8}{
\frac{\partial G_t}{\partial t}(y,\eta )+H_{\Im p}G_t+
{\cal O}((\nabla G_t)^2)=-\Re p(y,\eta ).
}
Now, $G_t\le 0$, so $(\nabla G_t)^2={\cal O}(G_t)$ and we obtain
\ekv{ev.9}
{
(\frac{\partial }{\partial t}+H_{\Im p})G_t+{\cal O}(G_t)=-\Re p,\ G_0=0.
}
Viewing this as a differential inequality along the integral curves of
$H_{\Im p}$, we obtain
\ekv{ev.10}
{
-G_t(\exp (tH_{\Im p})(\rho ))\asymp \int_0^t \Re p (\exp sH_{\Im
  p}(\rho ))ds,
}
for all $\rho =(y,\eta )\in \mathrm{neigh\,}(\rho _0,{\bf R}^{2n})$,
$\rho _0=(y_0,\eta _0)$.

Now, introduce the following assumption corresponding to the case (B) in Theorem \ref{in2},
\ekv{ev.11}{
H_{\Im p}^j (\Re p)(\rho _0)\cases{=0,\ j\le k-1\cr >0,\ j=k},
}
where $k$ necessarily is even (since $\Re p\ge 0$). We will work in a
sufficiently small neighborhood of $\rho _0$. Put
\ekv{ev.12}{J(t,\rho )=\int_0^t \Re p(\exp sH_{\Im p}(\rho ))ds,}
so that $0\le J(t,\rho )\in C^\infty (\mathrm{neigh\,}(0,\rho
_0),[0,+\infty [\times {\bf R}^{2n})$, and 
\ekv{ev.13}
{
\partial _t^{j+1} J(0,\rho _0)=
H_{\Im p}^j (\Re p)(\rho _0)\cases{=0,\ j\le k-1\cr >0,\ j=k}.
}
\begin{prop}\label{ev1}
Under the above assumptions, there is a constant $C>0$ such that 
\ekv{ev.14}
{
J(t,\rho )\ge \frac{t^{k+1}}{C},\ (t,\rho )\in
\mathrm{neigh\,}((0,\rho _0),]0,+\infty [\times {\bf R}^{2n}).
}
\end{prop}
\begin{proof}
Assume that (\ref{ev.14}) does not hold. Then there is a sequence
$(t_\nu ,\rho _\nu )\in [0,+\infty [\times {\bf R}^{2n}$ converging to
$(0,\rho _0)$ such that 
$$
\frac{J(t_\nu ,\rho _\nu )}{t_\nu ^{k+1}}\to 0,
$$
and since $J(t,\rho )$ is an increasing function of $t$
, we get 
$$
\sup_{0\le t\le t_\nu }\frac{J(t ,\rho _\nu )}{t_\nu ^{k+1}}\to 0.
$$
Introduce the Taylor expansion,
$$
J(t,\rho _\nu )=a_\nu ^{(0)}+a_\nu ^{(1)}t+...+a_\nu
^{(k+1)}t^{k+1}+{\cal O}(t^{k+2}),
$$
and define
$$
u_\nu (s)=\frac{J(t_\nu s,\rho _\nu )}{t_\nu ^{k+1}},\ 0\le s\le 1.
$$
Then, on the one hand,
$$
\sup_{0\le s\le 1}u_\nu (s)\to 0,\ \nu \to \infty ,
$$
and on the other hand,
$$
u_\nu (s)=\underbrace
{
\frac{a_\nu ^{(0)}}{t_\nu ^{k+1}}+\frac{a_\nu ^{(1)}}{t_\nu
  ^{k}}s+...+a_\nu ^{(k+1)}s^{k+1}}_{=:p_\nu (s)}+{\cal O}(t_\nu s^{k+2}),
$$
so
$$
\sup_{0\le s\le 1}p_\nu (s)\to 0,\ \nu \to \infty .
$$
The corresponding coefficients of $p_\nu $ have to tend to $0$, and in 
particular,
$$
a_\nu ^{(k+1)}=\frac{1}{(k+1)!}(\partial _t^{k+1}J(0,\rho _\nu )\to 0
$$
which is in contradiction with (\ref{ev.13}). \end{proof}

\par Combining (\ref{ev.10}) and Proposition \ref{ev1}, we get
\begin{prop}\label{ev2}
Under the assumption (\ref{ev.11}) there exists $C>0$ such that
\ekv{ev.15}
{
G_t(\rho )\le -\frac{t^{k+1}}{C},\ (t,\rho )\in
\mathrm{neigh\,}((0,\rho _0),[0,\infty [\times {\bf R}^{2n}).
}
\end{prop}

We can now return to the evolution equation for $\widetilde{P}$ and
the $t$-dependent weight $\Phi _t$ in (\ref{ev.0.17}). From (\ref{ev.15}),
(\ref{ev.3}), we get
\begin{prop}\label{ev3}
Under the assumption (\ref{ev.11}), we have
\ekv{ev.16}{
\Phi _t(x)\le \Phi _0(x)-\frac{t^{k+1}}{C},\ (t,x)\in \mathrm{neigh\,}
((0,x_0),[0,\infty [\times {\bf R}^{2n}).
}
\end{prop}

\section{The resolvent estimates}\label{re}
\setcounter{equation}{0}

Let $P$ be an $h$-pseudodifferential operator satisfying the general assumptions of the introduction.

Let $z_0\in (\partial \Sigma (p))\setminus \Sigma _\infty (p)$. We
first treat the case of Theorem \ref{in2} so that,
\ekv{re.1}
{
z_0\in i{\bf R},
}
\ekv{re.2}
{
\Re p(\rho )\ge 0 \hbox{ in }\mathrm{neigh\,}(p^{-1}(z_0),T^*X),
}
\ekv{re.3}
{
\forall \rho \in p^{-1}(z_0),\ \exists j\le k,\hbox{ such that }H_{\Im
p}^j\Re p(\rho )>0.
}
\begin{prop}\label{re1} $\exists C_0>0$ such that $\forall C_1>0$, $\exists C_2>0$ such that we have for $z,h$ as in (\ref{in.13.5}), $h<1/C_2$, $u\in C_0^\infty (X)$:
\eekv{re.4}
{&&
|\Re z|\Vert u\Vert\le C_0 \Vert (z-P)u\Vert,\hbox{ when }\Re z\le -h^{\frac{k}{k+1}},
}
{&&
h^{\frac{k}{k+1}}\Vert u\Vert\le C_0 \exp (\frac{C_0}{h}(\Re z)_+^{\frac{k+1}{k}}) \Vert (z-P)u\Vert,\hbox{ when }\Re z\ge -h^{\frac{k}{k+1}}.
}

\end{prop}
\begin{proof}
The required estimate is easy to obtain microlocally in the region
where $P-z_0$ is elliptic, so we see that it suffices to show the
following statement:

\par For every $\rho _0\in p^{-1}(z_0)$, there exists $\chi \in
C_0^\infty (T^*X)$, equal to 1 near $\rho _0$, such that for $z,h$ as in
(\ref{in.13.5}) and letting $\chi $ also denote a corresponding $h$-pseudodifferential operator, we have
\eekv{re.5}
{&&
|\Re z|\Vert \chi u\Vert\le C_0 \Vert (z-P)u\Vert +C_Nh^N\Vert u\Vert,\hbox{ when }\Re z\le -h^{\frac{k}{k+1}},
}
{&& \hskip -1cm
h^{\frac{k}{k+1}}\Vert \chi u\Vert\le C_0 \exp (\frac{C_0}{h}(\Re z)_+^{\frac{k+1}{k}}) \Vert (z-P)u\Vert +C_N h^N\Vert u\Vert,\hbox{ when }\Re z\ge -h^{\frac{k}{k+1}},
}
where $N\in {\bf N}$ can be chosen arbitrarily.

\par When $\Re z\le -h^{k/(k+1)}$ this is an easy consequence of the semi-classical sharp G\aa{}rding inequality (see for instance \cite{DiSj99}), so from now on we assume that $\Re z\ge -h^{k/(k+1)}$
.

\par If $T$ is an FBI transform and $\widetilde{P}$ denotes the
conjugated operator $TPT^{-1}$, it suffices to show that
\ekv{re.6}
{
\Vert u\Vert_{H_{\Phi _0}(V_1)}\le 
h^{-\frac{k}{k+1}}C_0 \exp (\frac{C_0}{h}(\Re z)_+^{\frac{k+1}{k}})
\Vert
(\widetilde{P}-z)u\Vert_{H_{\Phi _0}(V_2)}+{\cal O}(h^\infty )\Vert
u\Vert_{H_{\Phi _0}(V_3)}, 
}
$u\in H_{\Phi _0}(V_3)$, 
where $V_1\Subset V_2\Subset V_3$ are neighborhoods of $x_0$, given by 
$(x_0,\xi _0)=\kappa _T(\rho _0)\in \Lambda _{\Phi _0}$.

From Proposition \ref{ev3} and the fact that $\widetilde{U}(t):H_{\Phi
_0}(V_2)\to H_{\Phi _t}(V_1)$, we see that 
\ekv{re.7}
{
\Vert \widetilde{U}(t)u\Vert_{H_{\Phi _0}(V_1)}\le Ce^{-t^{k+1}/C}
\Vert u\Vert_{H_{\Phi _0}(V_2)}.
}
Choose $\delta >0$ small enough so that $\delta (k+1)<1$ and put 
\ekv{re.8}
{
\widetilde{R}(z)=\frac{1}{h}\int_0^{h^\delta }
e^{\frac{tz}{h}}\widetilde{U}(t) dt.
}

\par We shall verify that $\widetilde{R}$ is an approximate left
inverse to $\widetilde{P}-z$, but first we study the norm of
this operator in $H_{\Phi _0}$, starting with the estimate in ${\cal L}(H_{H_{\Phi _0}(V_2)},H_{H_{\Phi _0}(V_1)})$:
\ekv{re.9}
{
\Vert e^{\frac{tz}{h}}\widetilde{U}(t)\Vert \le C \exp \frac{1}{h}(t\Re
z-\frac{t^{k+1}}{C})
}
and notice that the right hand side is ${\cal O}(h^\infty )$
for $t=h^\delta $, since $\delta (k+1)<1$ and $h^{-1}\Re z\le {\cal O}(1)\ln \frac{1}{h}$.

\par We get 
\ekv{re.10}{\Vert \widetilde{R}(z)\Vert\le
 \frac{C}{h}\int_0^{+\infty
} \exp \frac{1}{h}(t\Re
z-\frac{t^{k+1}}{C})dt=\frac{C^{\frac{k+2}{k+1}}}{h^{\frac{k}{k+1}}}I(\frac{C^{\frac{1}{k+1}}}{h^{\frac{k}{k+1}}}\Re z),}
where
\ekv{re.11}
{
I(s)=\int_0^\infty e^{st-t^{k+1}}dt.
}
\begin{lemma}\label{re2}
We have
\ekv{re.12}{I(s)={\cal O}(1),\hbox{ when }|s|\le 1,}
\ekv{re.13}{I(s)=\frac{{\cal O}(1)}{|s|},\hbox{ when }s\le -1,}
\ekv{re.14}{I(s)\le {\cal O}(1)s^{-\frac{k-1}{2k}}
\exp \left(
  \frac{k}{(k+1)^{\frac{k+1}{k}}}s^{\frac{k+1}{k}}\right),\hbox{ when
}s\ge 1.
}
\end{lemma}
\begin{proof}
The first two estimates are straight forward and we concentrate on the
last one, where we may also assume that $s\gg 1$. A 
computation shows that the exponent $f_s(t)=st-t^{k+1}$ on $[0,+\infty
[$ has a unique critical point $t=t(s)=(s/(k+1))^{1/k}$ which is a
nondegenerate maximum,
$$
f_s''(t(s))=-k(k+1)^{\frac{1}{k}}s^{\frac{k-1}{k}},
$$
with critical value
$$
f_s(t(s))=\frac{k}{(k+1)^{\frac{k+1}{k}}}s^{\frac{k+1}{k}}.
$$
It follows that the upper bound in (\ref{re.14}) is the one we would
get by applying the formal stationary phase formula. 

\par Now 
$$
f_s''(t)=-(k+1)kt^{k-1}\lesssim f_s''(t(s)),\hbox{ for
}\frac{t(s)}{2}\le t<+\infty ,
$$
so $\int_{t(s)/2}^\infty e^{st-t^{k+1}}dt$ satisfies the required
upper bound.

\par On the other hand we have 
$$
f_s(t(s))-f_s(t)\ge \frac{s^{\frac{k+1}{k}}}{C},\hbox{ for }0\le t\le \frac{t(s)}{2},\ s\gg 1,
$$
so 
$$
\int_0^{\frac{t(s)}{2}} e^{st-t^{k+1}}dt\le {\cal O}(1) s^{\frac{1}{k}}\exp (f_s(t(s))-\frac{s^{\frac{k+1}{k}}}{C}),
$$
and (\ref{re.14}) follows.
\end{proof}

\par Applying this to (\ref{re.10}), we get
\begin{prop}\label{re3}
We have 
\ekv{re.15}{\Vert \widetilde{R}(z)\Vert\le
  \frac{C}{h^{\frac{k}{k+1}}},\quad |\Re z|\le {\cal O}(1)h^{\frac{k}{k+1}},}
\ekv{re.16}{\Vert \widetilde{R}(z)\Vert\le \frac{C}{|\Re z|},\quad -1\ll
  \Re z \le -h^{\frac{k}{k+1}},}
\ekv{re.17}{\Vert \widetilde{R}(z)\Vert\le \frac{C}{h^{\frac{k}{k+1}}}
\exp (C_k \frac{(\Re z)^{\frac{k+1}{k}}}{h}),\quad h^{\frac{k}{k+1}}\le \Re
z \ll 1.
}
\end{prop}

\par From the beginning of the proof of Lemma \ref{re2}, or more
directly from (\ref{re.9}), we see that 
$$
\Vert e^{\frac{tz}{h}}\widetilde{U}(t)\Vert \le C\exp
\frac{C_k}{h}(\Re z)_+^{\frac{k+1}{k}},
$$
which is bounded by some negative power of $h$, since we have imposed
the restriction $\Re z\le {\cal O}(1)(h\ln
\frac{1}{h})^{\frac{k}{k+1}}$. Working locally, we then see that
modulo a negligible operator,
$$
\widetilde{R}(z)(\widetilde{P}-z)\equiv \frac{1}{h}\int_0^{h^\delta }
e^{\frac{tz}{h}}(-h\partial _t-z)\widetilde{U}(t)dt\equiv 1,
$$
where the last equivalence follows from an integration by parts and the fact that the integrand is negligible for $t=h^\delta $.
Combining this with Proposition \ref{re3}, we get (\ref{re.6}), and
this completes the proof of Proposition \ref{re1}.
\end{proof}

\par We can now finish the 

\begin{proofof}
Theorem \ref{in2}. Using standard
pseudodifferential machinery (see for instance \cite{DiSj99}) we first
notice that $P$ has discrete spectrum in a neighborhood of $z_0$ and
that $P-z$ is a Fredholm operator of index 0 from ${\cal D}(P)$ to
$L^2$ when $z$ varies in a small neighborhood of $z_0$. On the other
hand, Proposition \ref{re1} implies that $P-z$ is injective and hence
bijective for $\Re z\le {\cal O}(k^{k/(k+1)})$ and we also get the
corresponding bounds on the resolvent.\end{proofof}

\begin{proofof} Theorem \ref{in1}: We may asume for simplicity that
  $z_0=0$ and consider a point $\rho _0\in p^{-1}(0)$. After
  conjugation with a microlocally defined unitary Fourier integral
  operator, we may assume that $\rho _0=(0,0)$ and that $dp(\rho
  _0)=d\xi _n$. Then from Malgrange's preparation theorem we get near
  $\rho =(0,0)$, $z=0$
\ekv{re.18}
{
p(\rho) -z=q(x,\xi ,z)(\xi _n+r(x,\xi ',z)),\ \xi '=(\xi _1,...,\xi _{n-1}),
}
where $q$, $r$ are smooth and $q(0,0,0)\ne 0$, and as in
\cite{DeSjZw}, we notice that either $\Im r(x,\xi ',0)\ge 0$ in a
neighborhood of $(0,0)$ or $\Im r(x,\xi ',0)\le 0$ in a
neighborhood of $(0,0)$. Indeed, otherwise there would exist sequences
$\rho _j^+$, $\rho _j^-$ in ${\bf R}^n\times {\bf R}^{n-1}$,
converging to $(0,0)$ such that $\pm \Im r(\rho _j^{\pm})>0$. It is
then easy to construct a simple closed curve $\gamma _j$ in a small
neighborhood of $\rho _0$, passing through the points $(\rho
_j^{\pm},0)$, such that the image of $\gamma _j$ under the map $(x,\xi
)\mapsto \xi _n+r(x,\xi ',0)$ is a simple closed curve in ${\bf
  C}\setminus 
\{ 0\}$, with winding number $\ne 0$. Then the same holds for the
image of $\gamma _j$ under $p$, and we see that ${\cal R}(p)$ contains
a full neighborhood of $0$, in contradiction with the assumption that
$0=z_0\in \partial \Sigma (p)$. 

In order to fix the ideas, let us assume that $\Im r\le 0$ near $\rho
_0$ when $z=0$, so that $\Re (i(\xi _n+r(x,\xi ',0)))\ge 0$. From
(\ref{re.18}), we get the pseudodifferential factorization
\ekv{re.19}
{
P(x,hD_x;h)-z=\frac{1}{i}Q(x,hD_x,z;h)\widehat{P}(x,hD_x,z;h),
}
microlocally near $\rho _0$ when $z$ is close to $0$. Here $Q$ and
$\widehat{P}$ have the leading symbols $q(x,\xi ,z)$ and $i(\xi
_n+r(x,\xi ',z))$ respectively.

\par We can now obtain a microlocal apriori estimate for $\widehat{P}$
as before. Let us first check that the assumption in (B) of Theorem \ref{in1} amounts to the statement that for $z=z_0=0$:
\ekv{re.20}
{
H^j_{\Re \widehat{p}}\Im \widehat{p}(\rho _0)>0
}
for some $j\in\{ 1,2,...,k\}$. In fact, the assumption in Theorem \ref{in1} (B) is obviously invariant under multiplication of $p$ by non-vanishing smooth factors, so we drop the hats and assume from the start that $p=\widehat{p}$ and $\Im p\ge 0$. Put $\rho (t)=\exp tH_p(\rho _0)$, $r(t)=\exp tH_{\Re p}(\rho _0)$ and let $j\ge 0$ be the order of vanishing of $\Im p(r(t)$ at $t=0$. From $\dot{\rho }(t)=H_p(\rho (t))$, $\dot{r}(t)=H_{\Re p}(r(t))$, we get 
$$
\frac{d}{dt}(\rho -r)=iH_{\Im p}(r)+{\cal O}(\rho -r),
$$
so 
$$
\rho (t)-r(t)=\int_0^t {\cal O}(\nabla \Im p(r(s))) ds.
$$
Here, if $p_2=\frac{1}{2i}(p-p^*)$ is the almost holomorphic extension of $\Im p$, we get 
\begin{eqnarray*}
&p^*(\rho (t))=ip_2(\rho (t))=&\\
&ip_2(r(t))+i\nabla p_2(r(t))\cdot (\rho (t)-r(t))+{\cal O}((\rho 
(t)-r(t))^2)=&\\
&ip_2(r(t))+i\nabla p_2(r(t))\cdot \int_0^t {\cal O}(\nabla p_2(r(s)))ds
+{\cal O}(1)(\int _0^t {\cal O}(\nabla p_2(r(s)))ds)^2.&
\end{eqnarray*} 
Here, $\nabla p_2(r(t))={\cal O}(p_2(r(t))^{1/2})={\cal O}(t^{j/2})$, so 
$p^*(\rho (t))=ia(\rho _0)t^j+{\cal O}(t^{j+1})$.

\par Then, if we conjugate with an FBI-Bargmann transform as above, we
can construct an approximation $\widetilde{U}(t)$ of $\exp
(-t\widetilde{\widehat{P}}/h)$, such that 
$$
\| \widetilde{U}(t)\| \le C_0 e^{(C_0t|z-z_0|-t^{k+1}/C_0)/h},
$$
when $|z-z_0|={\cal O}((h\ln \frac{1}{h})^{k/(k+1)})$.

\par From this we obtain a microlocal apriori estimate for
$\widehat{P}$ analogous to the one for $P-z$ in Proposition \ref{re1},
and the proof can be completed in the same way as for Theorem \ref{in2}.
\end{proofof}

\section{Examples}\label{ex}
\setcounter{equation}{0}
Consider 
\ekv{ex.1}
{
P=-h^2\Delta +iV(x),\ V\in C^\infty (X;{\bf R}),
}
where either $X$ is a smooth compact manifold of dimension $n$ or $X={\bf R}^n$. In the second case we assume that $p=\xi ^2+iV(x)$ belongs to a symbol space $S(m)$ where $m\ge 1$ is an order function. It is easy to give quite general sufficient condition for this to happen, let us just mention that if $V\in C_b^\infty ({\bf R}^2)$ then we can take $m=1+\xi ^2$ and if $\partial ^\alpha V(x)={\cal O}((1+|x|)^2)$ for all $\alpha \in {\bf N}^n$ and satisfies the ellipticity condition $|V(x)|\ge C^{-1}|x|^2$ for $|x|\ge C$, for some constant $C>0$, then we can take $m=1+\xi ^2+x^2$.

We have $\Sigma (p)=[0,\infty [+i\overline{V(X)}$. When $X$ is compact then $\Sigma _\infty (p)$ is empty and when $X={\bf R}^n$, we have $\Sigma _\infty (p)=[0,\infty [+i\Sigma _\infty (V)$, where $\Sigma _\infty (V)$ is the set of accumulation points at infinity of $V$. 

Let $z_0=x_0+iy_0\in \partial \Sigma (p)\setminus \Sigma _\infty (p)$.
\begin{itemize}
\item In the case $x_0=0$ we see that Theorem \ref{in2} (B) is applicable with $k=2$, provided that $y_0$ is not a critical value of $V$.
\item Now assume that $x_0>0$ and that $y_0$ is either the maximum or the minimum of $V$. In both cases, assume that $V^{-1}(y_0)$ is finite and 
that each element of that set is a non-degenerate maximum or minimum. Then Theorem \ref{in2} (B) is applicable to $\pm iP$. By allowing a more complicated behaviour of $V$ near its extreme points, we can produce examples where \ref{in2} (B) applies with $k>2$.
\end{itemize}

\par Now, consider the non-self-adjoint harmonic oscillator 
\ekv{ex.2}{
Q=-\frac{d^2}{dy^2}+iy^2
}
on the real line, studied by Boulton \cite{Bou} and Davies \cite{Da2}. Consider a large spectral parameter $E=i\lambda +\mu $ where $\lambda \gg 1$ and $|\mu |\ll \lambda $. The change of variables $y=\sqrt{\lambda }x$ permits us to identify $Q$ with 
$Q=\lambda P$, where $P=-h^2\frac{d^2}{dx^2}+ix^2$ and $h=1/\lambda \to 0$. 
Hence $Q-E=\lambda (P-(1+i\frac{\mu }{\lambda }))$ and Theorem \ref{in2} (B)
 is applicable with $k=2$. We conclude that $(Q-E)^{-1}$ is well-defined and of polynomial growth in $\lambda $ (which can be specified further) respectively ${\cal O}(\lambda ^{-1})$ when 
$$
\frac{\mu }{\lambda }\le C_1 (\lambda ^{-1}\ln \lambda )^{\frac{2}{3}}\hbox{ and } 
\frac{\mu }{\lambda }\le C_1\hbox{ respectively,}
$$ 
for any fixed $C_1>0$,
i.e. when 
\ekv{ex.3}
{
\mu \le C_1\lambda ^{\frac{1}{3}}(\ln \lambda )^{\frac{2}{3}}\hbox{ and }
\mu \le C_1\lambda ^{\frac{1}{3}} \hbox{ respectively.}
}

We end by making a comment about the Kramers--Fokker--Planck operator
\ekv{ex.4}
{
P=h y\cdot \partial_x-V'(x)\cdot h\partial _y+\frac{1}{2}(y-h\partial _y)\cdot (y+h\partial _y) 
}
on ${\bf R}^{2n}={\bf R}^n_x\times {\bf R}^n_y$, where $V$ is smooth and real-valued. The associated semi-classical symbol is 
$$
p(x,y;\xi ,\eta )=i(y\cdot \xi -V'(x)\cdot \eta )+\frac{1}{2}(y^2+\eta ^2)
$$
on ${\bf R}^{4n}$, and we notice that $\Re p_1\ge 0$. Under the assumption that the Hessian $V''(x)$ is bounded with all its derivatives, $|V'(x)|\ge C^{-1}$ when $|x|\ge C$ for some $C>0$, and that $V$ is a Morse function, F.~H\'erau, C.~Stolk and the author \cite{HeSjSt} showed among other things that the spectrum in any given strip $i[\frac{1}{C_1},C_1]+{\bf R}$ is contained in a half strip 
\ekv{ex.5}
{
i[\frac{1}{C_1},C_1]+[\frac{h^{2/3}}{C_2},\infty [
} 
for some $C_2=C_2(C_1)>0$ and that the resolvent is ${\cal O}(h^{-2/3})$ in the complementary halfstrip. (We refrain from recalling more detailed statements about spectrum and absence of spectrum in the regions where $|\Im z|$ is large and small respectively.)

The proof of this result employed exponentially weighted estimates based on the fact that $H_{p_2}^2p_1>0$ when $p_2\asymp 1$, $p_1\ll 1$. This is of course reminiscent of Theorem \ref{in2} (B) with $k=2$ or rather the corresponding result in \cite{DeSjZw}, but actually more complicated since our operator is not elliptic near $\infty $ and we even have that $i{\bf R}\setminus \{ 0\}$ is not in the range of $p$ but only in $\Sigma _\infty (p)$. It seems likely that the estimates on the spectrum of the KFP-operator above can be improved so that we can replace $h$ by $h\ln (1/h)$ in the confinement (\ref{ev.5}) of the spectrum of $P$ in the strip 
$i[1/C_1,C_1]+{\bf R}$ and that there are similar improvements for large and small values of $|\Im z|$. This would be obtained either by a closer look at the proof in \cite{HeSjSt} or by an adaptation of the proof above when $k=2$.

\end{document}